\documentclass[12pt,a4paper]{amsart}
\usepackage[utf8]{inputenc}
\usepackage{amsfonts}
\usepackage{amssymb}
\usepackage{amsmath}
\usepackage{amsthm}
\usepackage{mathrsfs}
\usepackage{cite}
\usepackage{enumitem}
\usepackage{tikz}
\usepackage{subfigure}
\usepackage{dsfont}

\theoremstyle{plain}
\newtheorem{thm}{Theorem}

\newtheorem{lem}{Lemma}
\newtheorem{prop}{Proposition} 
\newtheorem{cor}{Corollary}

\providecommand{\sm}{\setminus}
\providecommand{\N}{\mathbb{N}}
\providecommand{\R}{\mathbb{R}}
\providecommand{\Z}{\mathbb{Z}}

\providecommand{\T}{\mathbb{T}}
\providecommand{\I}{\mathcal{I}}
\providecommand{\J}{\mathcal{J}}
\providecommand{\cF}{\mathcal{F}}
\renewcommand{\L}{\mathcal{L}}
\providecommand{\eps}{\varepsilon}
\providecommand{\dx}{\:\mathrm{d}x}
\providecommand{\dxi}{\:\mathrm{d}\xi}

\providecommand{\dt}{\:\mathrm{d}t}

\providecommand{\dtx}{\:\mathrm{d}(t, x)}
\newcommand{\norm}[1]{\left\lVert #1 \right\rVert}

\renewcommand{\i}{\mathrm{i}}
\newcommand{\e}[1]{\mathrm{e}^{#1}}

\makeatletter
\newcommand{\pushright}[1]{\ifmeasuring@#1\else\omit\hfill$\displaystyle#1$\fi\ignorespaces}
\makeatother

\DeclareMathOperator{\dist}{dist}
\DeclareMathOperator{\supp}{supp}
\DeclareMathOperator{\spa}{span}
\DeclareMathOperator{\Real}{Re}

\newcounter{stepcount}
\newcounter{substepcount}[stepcount]

\newenvironment{steps}[0]{
\setcounter{stepcount}{0}
\setcounter{substepcount}{0}
\newcommand{\step}[1]{\stepcounter{stepcount}
\vspace*{0.2cm} \underline{\textsc{Step} \arabic{stepcount}}: \textit{##1} 

\vspace*{0.1cm}
}

}{ }

\setitemize{itemsep=+2pt}
\setenumerate{itemsep=+2pt}
\begin{document}

\allowdisplaybreaks

\title{Variational methods for breather solutions of Nonlinear Wave Equations}

\author{Rainer Mandel}
\address{R. Mandel\hfill\break
Karlsruhe Institute of Technology \hfill\break
Institute for Analysis \hfill\break
Englerstra{\ss}e 2 \hfill\break
D-76131 Karlsruhe, Germany}
\email{Rainer.Mandel@kit.edu}
\date{\today}
%

\author{Dominic Scheider}
\address{ D. Scheider\hfill\break
Karlsruhe Institute of Technology \hfill\break
Institute for Analysis \hfill\break
Englerstra{\ss}e 2 \hfill\break
D-76131 Karlsruhe, Germany}
\email{dominic.scheider@kit.edu}
\date{\today}
%

\begin{abstract}
We construct infinitely many real-valued, time-periodic breather solutions of the nonlinear wave equation
$$
	\partial_{tt} U - \Delta U = Q(x) |U|^{p-2} U
	\quad \text{on } \T \times \R^N
$$
 with suitable $N \geq 2$, $p > 2$ and localized nonnegative $Q$. These solutions are obtained from
 critical points of a dual functional and they are weakly localized in space.
 Our abstract framework allows to find similar existence results for the nonlinear Klein-Gordon equation
 and biharmonic wave equations.
\end{abstract}

\maketitle
\allowdisplaybreaks

\section{Introduction}

Breathers are real-valued, time-periodic and spatially localized solutions of nonlinear equations describing
the propagation of waves on $\R^N\times\R$ where $N\in\N$. The  existence of breather solutions appears to be
a rare phenomenon and up to now, most work in this area is related to the discussion of explicit examples such
as the famous sine-Gordon breather for the $(1+1)$-sine-Gordon equation~\cite{Abl_Method}. A
number of (in-)stability results for such
explicit breathers~\cite{AleMunPal_VariationalStructure,Ale_NLStab,AleMun_NLStability,AleFanMun_Akhmediev} is
available. Nonexistence results can be found in
\cite{KowMarMun_Nonexistence,Denzler_Nonpersistence,MunPon_Breathers}.
The construction of non-explicit breather solutions is a very difficult task. In papers by
Hirsch, Reichel~\cite[Theorem~1.3]{Hirsch} and Blank, Chirilus-Bruckner, Lescarret,
Schneider~\cite{Schneider} this was achieved for nonlinear wave equations of the form
\begin{equation}\label{eq:1+1}
  s(x)\partial_{tt}u - u_{xx} + q(x)u = f(x,u)
  \qquad (x,t\in\R)
\end{equation}
following two completely different approaches. The methods from~\cite{Schneider} come from spatial dynamics
and rely on center manifold reductions. For one very specific choice of periodic step functions $s,q$
(multiples of each other) and the nonlinearity  $f(x,u)= u^3$, the authors prove the existence of $L^\infty$-small periodic
breather solutions that are exponentially localized in space. The particular choice for $s$ and $q$ is motivated by the underlying
spectral theory of periodic Hill operators, also called Floquet theory. Using variational methods instead,
Hirsch and Reichel~\cite{Hirsch} proved the existence of (spatially) square integrable breather solutions
under appropriate assumptions on the nonlinearity. The latter include power-type nonlinearities
$f(x,u)=|u|^{p-1}u$ with $1<p<p^*$ for some $p^*$ depending on the choice of $s$ and $q$. Again, the
potentials $s,q$ are of very special form in order to ensure suitable spectral properties. More precisely, it
is required that for all $k\in\Z$ the  spectrum of the linear operator associated
with the $k$-th mode does not contain 0 in a uniform sense.
This makes it possible to have a strong localization in space. All of these results concern the case of one
spatial dimension $N=1$. The Bethe-Sommerfeld Conjecture about the number of gaps of periodic Schr\"odinger
operators (see \cite[Section 6.1.3]{Kuch_Overview}) suggests that the above approach can hardly be
generalized to higher space dimensions that we discuss next.

\medskip

In the case $N\geq 2$ we are aware of very few results. The first deals with
a semilinear curl-curl wave equation in $\R^3\times\R$ where $-u_{xx}$ is replaced by
$\nabla\times\nabla\times u$ in~\eqref{eq:1+1} and $u$ is a three-dimensional vector field on $\R^3$. Using
that this part in the equation actually vanishes for gradient fields,  Plum and
Reichel~\cite{PlumReichel_Breather} succeeded in proving the existence of exponentially localized breather
solutions via ODE methods for suitable radially symmetric coefficient functions $s,q$ and power-type nonlinearities $f$. As far as we know, this is the only result
dealing with strongly localized breathers in higher dimensions, i.e., $U(t,\cdot)\in L^2(\R^N)$ for almost
all $t\in\R$.
Recently, the second author suggested a new construction of (even in time)
breathers~\cite{scheider2020breather} for the cubic Klein-Gordon equation that we will refer to as weakly localized in space. Those satisfy
$U(\cdot,t)\in L^q(\R^N)$ for almost all $t\in\R$ for some $q>2$ and we believe that in general
$U(t,\cdot)\notin L^2(\R^N)$ holds due to rather small decay rates at infinity, presumably $U(t,x)\sim |x|^{(1-N)/2}$ as
$|x|\to\infty$. That approach relies on the $L^p$-theory for Helmholtz
equations on $\R^N$ and bifurcation techniques allow to prove the existence of infinitely many branches consisting of
polychromatic radially symmetric breather solutions that emanate from a nontrivial stationary solution of the
problem. The solutions are of the form

\begin{equation}\label{eq:FourierAnsatzU}
  U(t,x)= \sum_{k\in \I_s} \e{\i kt}u_k(x),\qquad \I_s\subseteq \Z
\end{equation}
with radially symmetric Fourier modes $u_k=\overline{u_{-k}}$ infinitely many of which are
non-zero. Imposing radial symmetry is of course a significant restriction.

\medskip

In this paper we propose a  dual variational approach for the construction of weakly localized breather
solutions of abstract nonlinear wave equations without any symmetry assumptions on the coefficients.
This approach goes back to Br\'{e}zis, Coron, Nirenberg~\cite{Brezis_vibes} and Rabinowitz~\cite{Rab_vibes} in
the context of classical nonlinear wave equations in bounded domains. Dual variational methods have the advantage
that they overcome the strong indefiniteness of the corresponding functional in the classical variational formulation.
They apply in the case $N\geq 2$ and do not rely on specific ``nonstandard'' choices for the elliptic part of
the wave operator that seem to be needed for the construction of strongly localized breathers via classical
variational methods when $N=1$, see~\cite{Hirsch}. The dual variational method is typically implemented
in function spaces different from $L^2$ or $H^1$ so that the constructed solutions do not necessarily give rise to
strongly localized breathers. In fact, we shall prove the existence of weakly localized breathers only. It
is entirely unclear whether strongly localized breathers of nonlinear wave equations exist in the case $N\geq
2$.

 \medskip

We prove the existence of weakly localized breathers for nonlinear wave equations of the
form
\begin{equation}\label{eq_general}
	 \partial_{tt}  U  + \L  U = Q(x)|U|^{p-2}U
	\quad \text{on }   \R\times\R^N .
\end{equation}
 Since we want~\eqref{eq_general} to hold in a distributional sense that we will  make precise
below, we only need to assume that there is some $m\in\N$ and a suitable $q\in [p,\infty]$ such that
$\L:W^{m,\infty}(K)\to L^{q'}(\R^N)$ is a bounded linear map for all compact $K\subset\R^N$, see assumption
(A1) below. Here and in the following, $q'$ denotes the H\"older conjugate exponent of $q$ defined by
$\frac{1}{q}+\frac{1}{q'}=1$.
As a model case one may have in mind $\L=-\Delta$. In contrast to~\cite{scheider2020breather} we can deal
with nonradial $Q$ and obtain the existence of an unbounded sequence of breathers.
In order to avoid ``bad'' modes, we look for breather solutions that enjoy additional symmetry properties with
respect to  time. To include such symmetries in our analysis we introduce a parameter $s\in\{1,\ldots,5\}$
that stands for
\begin{itemize}
\item[] ($s=1$)\quad no additional symmetry,
\item[] ($s=2$)\quad $U(t,x) = U(-t,x)$, i.e., $U$ is even in time, 
\item[] ($s=3$)\quad $U(t,x) = - U(-t,x)$, i.e., $U$ is odd in time,  
\item[] ($s=4$)\quad $U(t+\pi, x) = U(t,x)$, i.e., $U$ is $\pi$-periodic,  
\item[] ($s=5$)\quad $U(t+\pi, x) = -U(t,x)$, i.e., $U$ is $\pi$-antiperiodic.  
\end{itemize}
For those symmetries the relevant modes $k\in\I_s\subset\Z$  in the corresponding Fourier expansions
come from the sets
\begin{equation}\label{eq:defn_modes}
 \I_1:=\Z,\qquad \mathcal I_2 := \Z,\qquad \mathcal I_3 :=\Z\sm\{0\},\qquad \mathcal I_4:= 2\Z,\qquad
 \I_5:= 2\Z+1.
\end{equation}
Accordingly, we look for functions $u_k=\overline{-u_k}$ ($k\in\I_s)$
in order to ensure that the solution $U$ given by~\eqref{eq:FourierAnsatzU} is
real-valued. In the case $s=2$ resp. $s=3$ observe that the symmetry assumption
even requires  $u_k$ resp. $iu_k$ to be real-valued.  Regarding $s=4$ let us
mention that general periods $T>0$ can be discussed, as we will explain in Section~\ref{sec:gen}. We will assume that the following conditions are satisfied for
$k\in\I_s$:
\begin{itemize}
  \item[(A1)]  There are
  bounded symmetric operators $\mathcal{R}_k:L^{q'}(\R^N)\to L^q(\R^N)$ for some $q\in [p,\infty]$ that satisfy $\mathcal R_k=\mathcal
  R_{-k}$, $\norm{\mathcal{R}_k} \leq C (k^2 + 1)^{-\frac{\alpha}{2}}$ for some $\alpha > 1 - \frac{2}{p}$ as
  well as
  $$
  \int_{\R^N} \mathcal{R}_k f \cdot (\L-k^2) \phi \dx = \int_{\R^N} f \phi\dx
  \qquad\text{for all }f\in L^{q'}(\R^N),\;\phi\in C_c^\infty(\R^N)
$$
 where, for some $m\in\N$, $\L:W^{m,\infty}(K)\to L^{q'}(\R^N)$ is a bounded linear map for all compact
  $K\subset\R^N$.
 \item[(A2)] $Q \in L^{q/(q-p)}(\R^N),Q\geq 0, Q\not\equiv 0$ and the linear operators
 $v\mapsto \mathcal{R}^Q_k [v] := Q^{1/p} \: \mathcal{R}_k [Q^{1/p} v]$ are compact from $L^{p'}(\R^N)$ to
 $L^p(\R^N)$.
 \item[(A3)] There are $\omega_k \in L^{p'}(\R^N)$ with
 $
	\int_{\R^N} \omega_k \mathcal{R}^Q_k[\omega_k] \dx > 0.
$
\end{itemize}
 Our convention is that functions belonging to
$L^r(\R^N),1\leq r\leq \infty$, are real-valued and $\mathcal{R}_k$ is extended to complex-valued
functions by linearity, i.e., $\mathcal{R}_k(f+ig):= \mathcal R_k f + i\mathcal R_k g$ for $f,g\in
L^{q'}(\R^N)$.

\medskip

We briefly comment on these assumptions. The operators $\mathcal{R}_k$ from (A1)  can be interpreted as
distributional right inverses of $\L -k^2$ that may even exist when classical inverses are not
available. This is for instance the case in our model example $\L=-\Delta$ as we will show in
Lemma~\ref{lem_LAP}. The symmetry assumption means
$$
  \int_{\R^N} \mathcal{R}_k f\cdot g \dx
  = \int_{\R^N}  f\cdot \mathcal{R}_k g \dx
  \quad\text{for all }f,g\in L^{q'}(\R^N).
$$
The growth bound on the norms ensures the convergence of Fourier series in topologies that are suitable for
our analysis. Assumption (A2) is  needed for our dual variational approach, notably for the verification
of the Palais-Smale condition for the dual functional $J$ associated with~\eqref{eq_general} that we will
introduce later. Notice that $Q \in L^{q/(q-p)}(\R^N),Q\geq 0$ and
(A1) already imply the boundedness of $\mathcal R_k^Q$ as a map from $L^{p'}(\R^N)$ to $L^p(\R^N)$,
see~\eqref{eq:Bound:RQ}, but compactness requires for some decay of $Q$ at infinity. In contrast, (A2) is
often violated in the important special case $Q\equiv 1$ so that our approach does   not apply in
this case. We refer to Section~\ref{sec:gen} for further observations about this case and other more general
nonlinearities.  Finally, (A3) is a technical assumption that holds in many
applications. It is for instance satisfied if $Q$ is positive and for all $k\in\I_s$ there are test functions
$\phi_k\in C_c^\infty(\R^N)$ such that the inequality $\int_{\R^N} \phi_k ( \L -k^2)\phi_k\dx>0$ holds.
Indeed, choosing $\omega_k:=\omega_k^\delta$ for sufficiently small $\delta>0$ where
$\omega_k^\delta := Q^{-1/p}1_{Q>\delta} ( \L -k^2)\phi_k\in L^{p'}(\R^N)$ we obtain
\begin{align*}
  \lim_{\delta\to 0^+} \int_{\R^N} \omega_k^\delta \mathcal{R}^Q_k[\omega_k^\delta] \dx
  &= \lim_{\delta\to 0^+} \int_{\R^N} \big(1_{Q>\delta} (\L -k^2)\phi_k\big) \mathcal{R}_k[ 1_{Q>\delta}
  (\L -k^2)\phi_k] \dx
  \\
  &= \int_{\R^N} \big((\L -k^2)\phi_k \big)\mathcal{R}_k[(\L -k^2)\phi_k] \dx \\
  &= \int_{\R^N} \phi_k (\L -k^2)\phi_k\dx >0.
\end{align*}
Hence, (A3) holds for positive $Q$  provided that $\L$
is a uniformly elliptic differential operator on $\R^N$, say of the form $\sum_{|\alpha|\leq 2m}
a_\alpha(x)\partial^\alpha$ with locally integrable coefficients $a_\alpha$ and
$\sum_{|\alpha|=2m} a_\alpha(x) (i\xi)^\alpha \geq c|\xi|^{2m}$ for all
$x,\xi\in\R^N$ and some $c>0$.
Indeed, in that case one may choose $\phi_k=\phi(t_k\cdot)$ for some test function $\phi\in
C_c^\infty(\R^N)$ and $t_k>0$ sufficiently large, see the proof of Corollary~\ref{cor:CossettiMandel} for the
details in the case $\L=(-\Delta)^2$.  In Section~\ref{sec:App} we will see a number of
settings where all our assumptions hold.

\medskip

Assuming (A1)-(A3) for all modes $k\in\I_s$ we are going to prove the existence of breather solutions $U$
of~\eqref{eq_general} in the Banach space $L^q({\R^N}, L^p_s(\T))$ consisting of all elements of
$L^q({\R^N},L^p(\T))$ having the symmetry indexed by $s$.
The norm on these spaces is given by  
$$
  \|W\|_{L^q({\R^N}, L^p(\T))}
  := \norm{  \norm{W(\cdot,x)}_{L^p(\T)} }_q
  = \left( \int_{{\R^N}}  \left(\int_\T |W(t,x)|^p\dt\right)^{q/p} \dx\right)^{1/q},
$$
where $\|\cdot\|_q=\|\cdot\|_{L^q(\R^N)}$ is the standard norm in $L^q(\R^N)$
and $\T\simeq [0,2\pi]$ stands for the torus. More precisely, we will speak of $2\pi$-periodic real-valued  distributional breather solutions since we
require these functions to solve equation~\eqref{eq_general} in the following sense:
\begin{align}\label{eq_weaksolution}
	\int_{\T \times {\R^N}}  U \: (\partial_{tt}  - \L) \Phi \dtx
	= \int_{\T \times {\R^N}} Q(x)|U|^{p-2}U \: \Phi \dtx
	\quad \forall \Phi \in C_c^\infty({\R^N},C^\infty(\T)).
\end{align}
Here, $\Phi \in C_c^\infty({\R^N},C^\infty(\T))$ means that there is a compact subset $K\subseteq{\R^N}$
such that $\Phi:\T\times{\R^N}\to\R$  is smooth, $2\pi$-periodic
in time and the support of $\Phi(t,\cdot)$ is contained in $K$ for all $t\in\T$.
Our main result is the following.

\begin{thm}\label{thm_general}
  Assume $N\in\N$,  $2<p<\infty$ and (A1)-(A3) for all
  $k\in\I_s$ where $s\in\{1,\ldots,5\}$.
  Then the nonlinear wave equation~\eqref{eq_general} admits an unbounded sequence of
  $2\pi$-periodic distributional real-valued  breather solutions $U_j \in L^q({\R^N}, L^p_s(\T))$, $j\in
  \N_0$, in the sense of~\eqref{eq_weaksolution}.
\end{thm}

 We add that our breather solutions are either constant or polychromatic. The latter means
 that at least two Fourier modes $u_k,u_l$ with $k\neq l,k,l\in\N_0$ of the solution are non-zero. Indeed,
 plugging the ansatz $U(t,x) =\cos(kt)u_k(x)$ or $U(t,x) =\sin(kt)u_k(x)$  and nontrivial
 functions $u_k$ into~\eqref{eq_general} one infers $k=0$, so $U$ is necessarily constant
 in time. In our applications below this will be avoided by choosing $s\in\{3,5\}$ because of $0\notin \I_s$.
 In the case $p\in\{3,4,\ldots\}$ we can conclude as in~\cite[Theorem~1~(iii)]{scheider2020breather}
 that nonconstant in time breathers have in fact infinitely many nontrivial modes, which we believe to be the
 typical situation.
 Concerning the regularity of our breather solutions, we point out that, proceeding as in the proof of
 Proposition~\ref{prop_waveresolvent} below, it is possible to prove (local) Sobolev-regularity of solutions
 provided that estimates of the form
 $$
  \norm{\mathcal{R}_k f}_{W^{m,r}(K)}  \leq C (k^2 + 1)^{-\frac{\tilde\alpha}{2}}
  (\|f\|_{s_1}+\|f\|_{s_2})
$$
hold for suitable $K\subset\R^N$, $m\in\N$, $s_1,s_2,r\in [1,\infty]$ and  $\tilde\alpha>0$ such
that the corresponding Fourier series converge. However, we could not verify such an estimate for
large enough $\tilde \alpha>0$ in our applications, so we have to leave the regularity issue as an open
problem. We expect that in the case $\L=-\Delta$ the spatial decay of our breather solutions is $U(t,x)\sim
|x|^{(1-N)/2}$ and hence similar to the spatial decay of monochromatic complex-valued  solutions of the form
$\e{\i kt}u(x)$ that are modeled by a single Helmholtz equation instead of infinitely many coupled ones.
Similarly, we expect that the breathers oscillate in the sense that they have infinitely many nodal domains.
Given that the physical model requires for real-valued solutions, our results indicate that complex-valued
waves of the form $\e{\i kt}u(x)$ provide a reasonable simplified model for breather solutions.

 \medskip

 We outline how this paper is organized. In Section~\ref{sec:App} we show how
 Theorem~\ref{thm_general} applies in concrete situations. In particular, we prove the existence of
 infinitely many breathers of nonlinear wave equations and Klein-Gordon equations on $\R^N$. Moreover, we
 indicate further possible generalizations of our approach. In Section~\ref{sec:ProofThm} we motivate our variational approach and
 present the proof of Theorem~\ref{thm_general} relying on the technical results contained  in the
 Propositions~\ref{prop_waveresolvent}--\ref{prop_wavesolution}. The proofs of the latter are presented
 in Section~\ref{sec:ProofAux}.

\section{Applications and Examples} \label{sec:App}

\subsection{Breather Solutions for the Wave Equation on $\R^N$}

We show that Theorem~\ref{thm_general} applies to classical nonlinear wave equations on $\R^N$
with power-type nonlinearities
\begin{equation}\label{eq_wave}
	\partial_{tt}  U - \Delta U = Q(x) |U|^{p-2} U
	\qquad \text{on } \T \times \R^N.
\end{equation}
To verify (A1)--(A3) we need distributional right  inverses for operators of the form
$-\Delta-k^2$ for  $k\in\I_s$ and suitable $s\in\{1,\ldots,5\}$. From  \cite[Theorem~2.3]{Kenig},
\cite[Theorem~6]{Gutierrez} $(N\geq 3)$ and~\cite[Theorem~2.1]{Evequoz_plane} ($N = 2$) we infer that   the operators
$$
  \mathcal R_k f:= \lim_{\eps\to 0^+} \Real\left[\cF^{-1}\left(\frac{\cF f}{|\cdot|^2-k^2-i\eps}
  \right)\right]
$$
are suitable for that purpose. Here, $\cF$ denotes the Fourier transform in $\R^N$. For the asymptotics with
respect to $k$ in the estimate~\eqref{eq:KRS}  we refer to~\cite[Theorem~2.3]{Kenig} ($N\geq 3$) and
inequality~(8) in \cite{Frank_EVBounds} ($N=2$).

\begin{lem}[Kenig, Ruiz, Sogge] \label{lem_LAP}
  Let $N\in\N,N\geq 3$ and assume $\frac{2(N+1)}{N-1}\leq r\leq \frac{2N}{N-2}$.
  For every $k\in\Z\sm\{0\}$ the  operator  $\mathcal{R}_k:L^{r'}(\R^N)\to L^r(\R^N)$ is a
  bounded and symmetric distributional right  inverse of $-\Delta-k^2$ satisfying
  \begin{equation}\label{eq:KRS}
    \|\mathcal R_k f\|_r \leq C |k|^{-2+\frac{N}{r'}-\frac{N}{r}}\|f\|_{r'}
  \end{equation}
  for some $C>0$. In the case $N= 2$ the same holds for $6\leq r<\infty$.
\end{lem}

Notice that Lemma~\ref{lem_LAP} can also be seen as a special case of Lemma~\ref{lem_LAP_CosMan} below
and thus of~\cite[Theorem~4]{CosMan}, so we do not provide a proof here. 
We stress that this result does not provide a distributional right  inverse for $-\Delta$,
which is why we have to consider symmetries that exclude the zero mode. This and~\eqref{eq:defn_modes}
motivate the choice $s\in\{3,5\}$, so that we obtain the existence of infinitely many
odd-in-time $2\pi$-periodic breathers and infinitely many $\pi$-antiperiodic breathers.

\begin{cor}[The Wave Equation]\label{cor_wave}
Assume $N\in\N,N\geq 2$ and $Q\in L^\frac{q}{q-p}(\R^N),Q\geq 0, Q\not\equiv 0$ where $p,q$
satisfy
\begin{align*}
	&2 < p < \frac{2(N+1)}{N-1},
	\qquad
	 \frac{2(N+1)}{N-1} < q < \frac{2Np}{(N-1)p - 2}.
\end{align*}
Then, for $s\in\{3,5\}$, the nonlinear  wave equation~\eqref{eq_wave} admits an unbounded sequence of
$2\pi$-periodic real-valued  distributional breather solutions $U_j \in  L^q(\R^N;L^p_s(\T)), \, j \in \N_0$.
\end{cor}

\textbf{Proof of Corollary~\ref{cor_wave}.}
We verify the assumptions (A1) - (A3) for $\mathcal{L} = -\Delta$. As indicated above,
the choice $s\in\{3,5\}$ implies $\I_s\subset\Z\sm\{0\}$ so that $k\in\I_s$ implies $|k|\geq 1$. In
particular, the previous lemma applies and yields real-valued, bounded, symmetric linear operators $\mathcal
R_k$ that are distributional right  inverses of $-\Delta-k^2$ and satisfy
$$
  \|\mathcal R_k f\|_q \leq C(k^2+1)^{-\alpha/2}  \|f\|_{q'}
  \qquad (k\in\I_s)
$$
for $\alpha=2-\frac{N}{q'}+\frac{N}{q}$. Here we have used that our assumptions
imply $\frac{2(N+1)}{N-1}\leq q<\frac{2N}{N-2}$. From $q < \frac{2Np}{(N-1)p - 2}$ we moreover infer
$\alpha>1-\frac{2}{p}$. So assumption (A1) holds. The compactness of the Birman-Schwinger operator $\mathcal R_k^Q$ is proved as
in~\cite[Lemma~4.1]{EvequozWeth} $(N\geq 3)$ resp.~\cite[Section~3]{Evequoz_plane} $(N=2)$ or \cite[Lemma~3.1]{GolSch}.
(We will provide more details in the proof of Corollary~\ref{cor_wave_frac} below.) Taking
$\omega_k:= v_0(k\cdot)$ for the function $v_0$ from~\cite[Lemma~4.2~(ii)]{EvequozWeth} we find that (A3)
holds as well. Hence, Theorem~\ref{thm_general} yields the existence of an unbounded sequence of
distributional breathers in $L^q(\R^N;L^p_s(\T))$.
 \hfill $\square$

\subsection{Breather Solutions for the Klein-Gordon Equation}

We study the nonlinear  Klein-Gordon equation
\begin{equation}\label{eq_kleingordon}
	\partial_{tt}  U - \Delta U + m^2 U = Q(x) |U|^{p-2} U
	\qquad \text{on } \T \times \R^N.
\end{equation}
Much like for the wave equation, we deduce from Theorem~\ref{thm_general} the following

\begin{cor}[The Klein-Gordon Equation]\label{cor_kleingordon}
Assume $N\in\N,N\geq 2$, $m>0$ and $Q\in L^\frac{q}{q-p}(\R^N),Q>0 $
where $p,q$ satisfy
\begin{align*}
	 2 < p < \frac{2(N+1)}{N-1},
	\qquad
	 \frac{2(N+1)}{N-1} < q < \frac{2Np}{(N-1)p - 2}.
\end{align*}
Then the nonlinear  Klein-Gordon equation~\eqref{eq_kleingordon} admits an unbounded sequence of
$2\pi$-periodic real-valued  breather solutions  $U_j \in  L^q(\R^N;L^p_s(\T)), \, j \in \N_0$.
 Here, $s$ can be chosen as follows:
\begin{itemize}
\item[(i)] if $m \not\in \N$, then $s\in\{1,\ldots,5\}$,
\item[(ii)] if $m \in 2\N - 1$, then $s=4$ ($\pi$-periodic breathers),
\item[(iii)] if $m \in 2\N$, then $s=5$ ($\pi$-antiperiodic breathers).
\end{itemize}
\end{cor}

Since the proof is very much the same as for the wave equation, we omit it.  Let us
remark that our choice of $s$ again ensures that we avoid the modes $k\in\I_s$ with $m^2-k^2=0$. In the study
of the operators $- \Delta + m^2 - k^2$, there may now occur a finite number of operators ($k \in \I_s$ with $k^2 < m^2$)
with  classical $L^2$-inverses given by a
convolution with positive exponentially decaying kernels, see (2.21) in~\cite{Kenig}. The mapping
properties of these well-understood Bessel potential operators are in fact
much better than the ones mentioned in Lemma~\ref{lem_LAP} because all $r\in [2,\frac{2N}{N-2}]$ resp.
$r\in [2,\infty)$ are allowed in~\eqref{eq:KRS} if $N\geq 3$ resp. $N=2$.
Moreover, $-\Delta + m^2-k^2$ is uniformly elliptic so that the arguments presented in
the  Introduction imply the validity of (A2) and (A3) under the assumption $Q>0$. 

\subsection{Breather Solutions for Fractional and Biharmonic Wave Equations}

We consider the problem
\begin{equation}\label{eq_wavefrac}
	\partial_{tt}  U  + (- \Delta)^\gamma U = Q(x) |U|^{p-2} U
	\qquad \text{on } \T \times \R^N
\end{equation}
for general $\gamma>\frac{N}{N+1}$.
As in the case of the classical wave equation one finds distributional right  inverses of $(-\Delta)^\gamma-k^2$ with
the aid of the Limiting Absorption Principle that allows to make sense of the limits
$$
  \mathcal R_k^\gamma f:= \lim_{\eps\to 0^+}
  \Real\left[\cF^{-1}\left(\frac{\cF f}{|\cdot|^{2\gamma}-k^2-i\eps}
  \right)\right].
$$
This follows from a result by Huang, Yao, Zheng~\cite[Corollary~3.2]{Huang_LAP}.

\begin{lem}[Huang, Yao, Zheng]\label{lem_LAP_fractional}
  Let $N\in\N,N\geq 3$ and assume $\frac{2(N+1)}{N-1}\leq r<\frac{2N}{(N-2\gamma)_+}$, $\gamma>
  \frac{N}{N+1}$.
  For every $k\in\Z\sm\{0\}$ the  operator  $\mathcal{R}_k^\gamma:L^{r'}(\R^N)\to L^r(\R^N)$ is a
   bounded and symmetric distributional  right  inverse of $(-\Delta)^\gamma-k^2$ satisfying
  $$
    \|\mathcal R_k^\gamma f\|_r \leq |k|^{-2+\frac{N}{\gamma r'}-\frac{N}{\gamma r}}\|f\|_{r'}.
  $$
\end{lem}

The previous lemma provides the existence of distributional right  inverses for all $\gamma>
\frac{N}{N+1}$. Notice that this restriction on $\gamma$ is needed to ensure
$\frac{2(N+1)}{N-1}<\frac{2N}{(N-2\gamma)_+}.$ 
As in the case of the Laplacian we expect a similar result to hold in the two-dimensional case $N=2$. In
the following result we apply Lemma~\ref{lem_LAP_fractional} in order to prove the existence of
breathers to fractional nonlinear wave equations just as in the case $\gamma=1$ discussed in
Corollary~\ref{cor_wave}. We stress that this includes the case $\gamma=2$ of biharmonic nonlinear wave
equations.

\begin{cor}[Fractional Wave Equations]\label{cor_wave_frac}
Assume $N\in\N,N\geq 3,\gamma>\frac{N}{N+1}$ and $Q\in L^\frac{q}{q-p}(\R^N),Q> 0$ where
$p,q$ satisfy
\begin{align*}
	&2 < p < \frac{2\gamma(N+1)}{((2-\gamma)N-\gamma)_+},
	\qquad
	 \frac{2(N+1)}{N-1} < q < \frac{2Np}{((N-\gamma)p - 2\gamma)_+}.
\end{align*}
Then, for $s\in\{3,5\}$, the nonlinear  fractional wave equation~\eqref{eq_wavefrac} admits an unbounded
sequence of $2\pi$-periodic real-valued  distributional breather solutions $U_j \in L^q(\R^N,L^p_s(\T)), \,
j \in \N_0$.
\end{cor}
\textbf{Proof of Corollary~\ref{cor_wave_frac}.}  As in the proof of Corollary~\ref{cor_wave} the previous
lemma yields assumption~(A1) for $\alpha=2-\frac{N}{\gamma q'}+\frac{N}{\gamma q}$.
 For the verification of (A3) we follow~\cite[Lemma~3.1]{MaMoPe_Osc}. We define
 $\omega_k^\delta:=Q^{-1/p}\tilde\omega_k 1_{Q\geq \delta}\in L^{p'}(\R^N)$ for $\delta>0$  and the support
 of $\cF(\tilde\omega_k)$ is contained in $\{\xi\in\R^N:|\xi|^{2\gamma}>k^2\}$. Then the above definition for $\mathcal R_k^\gamma$ implies
 \begin{align*}
    \lim_{\delta\to 0^+}
    \int_{\R^N} \omega_k^\delta Q^{1/p}\mathcal R_k^\gamma[Q^{1/p}\omega_k^\delta] \dx
    &= \int_{\R^N} \tilde \omega_k \mathcal R_k^\gamma[ \tilde \omega_k] \dx \\
    &= \lim_{\eps\to 0} \Real\left(\int_{\R^N}
     \frac{|\cF(\tilde \omega_k)|^2}{|\xi|^{2\gamma}-k^2-i\eps} \dxi \right) \\
    &=  \int_{\R^N} \frac{|\cF(\tilde \omega_k)|^2}{|\xi|^{2\gamma}-k^2} \dxi
    > 0.
 \end{align*}
  So choosing $\delta>0$ sufficiently small and $\omega_k=\omega_k^\delta\in L^{p'}(\R^N)$  yields (A3).

  \medskip

  The verification of (A2) is standard for classical
  Schr\"odinger operators of second order. In order to see that in the fractional case nothing really changes,
  we repeat the main arguments here. In view of
  $$
    \|Q^{1/p}\mathcal R_k^\gamma[Q^{1/p}v]\|_p
    \leq \|Q\|_{\frac{q}{q-p}}^{1/p} \|\mathcal R_k^\gamma[Q^{1/p}v]\|_q
  $$
  it suffices to prove that $v\mapsto \mathcal R_k^\gamma[\Gamma v]$ is compact from $L^{p'}(\R^N)$ to
  $L^q(\R^N)$ where $\Gamma:=Q^{1/p}$. We may without loss of generality assume that $\Gamma$ is bounded with
  compact support.  Indeed, choosing $\Gamma_n\to \Gamma$ in $L^{\frac{pq}{q-p}}(\R^N)$ with $\Gamma_n$
  bounded and compact support, we find
  $$
    \|\mathcal R_k^\gamma[\Gamma v] - \mathcal R_k^\gamma[\Gamma_nv]\|_q
    \leq \|\mathcal R_k^\gamma[(\Gamma-\Gamma_n)v]\|_q
    \leq \|\mathcal R_k^\gamma\|_{q'\to q} \underbrace{\|\Gamma-\Gamma_n\|_{\frac{pq}{q-p}}}_{\to 0
    \text{ as }n\to\infty} \|v\|_{p'}.
  $$
  Having proved that  $v\mapsto \mathcal R_k^\gamma[\Gamma_n v]$ is compact for each $n\in\N$ we can thus
  conclude that $v\mapsto \mathcal R_k^\gamma[\Gamma v]$ is compact as the limit of compact operators with
  respect to the uniform operator topology. So it remains to prove the compactness of
  $v\mapsto \mathcal R_k^\gamma[\Gamma v]$ from $L^{p'}(\R^N)$ to $L^q(\R^N)$ assuming that $\Gamma$ is
  bounded with compact support.

  \medskip

  Let $B\subset\R^N$ be any bounded ball. The compactness of $v\mapsto \chi_B \mathcal R_k^\gamma[\Gamma v]$
  follows from the fractional Rellich-Kondrachov Theorem, see~\cite[Corollary~7.2]{Hitchhiker}.
  By the same argument as above, it remains to show $\|\chi_{\R^N\setminus B}\mathcal R_k^\gamma[\Gamma
  \cdot]\|_{p'\to q} \to 0$ as $B\nearrow \R^n.$ To this end we use
  \begin{equation*}
    \mathcal R_k^\gamma f   = G_k^\gamma\ast f,\qquad\text{where }
     G_k^\gamma(z)
     := \lim_{\eps\to 0^+} \Real\left[\cF^{-1}\left(\frac{1}{|\cdot|^{2\gamma}-k^2-i\eps} \right)(z) \right].
  \end{equation*}
  The formulas (3.8),(3.8') in~\cite[Corollary~3.2]{Huang_LAP} show that the kernel function
  satisfies
    $|G_k^\gamma(z)|\leq C_k|z|^{\frac{1-N}{2}} \quad\text{if }|z|\geq 1$
  for some $C_k>0$. Hence, for $M:= \supp(\Gamma)$ and $x\in\R^N$ such that
  $\dist(x,M)\geq 1$ we have
\begin{equation*}
	|\mathcal R_k^\gamma[\Gamma v](x)|
	\leq C_k \int_M |x-y|^\frac{1-N}{2}|\Gamma(y)| |v(y)|\, \mathrm{d}y
	\leq \tilde C_k |x|^\frac{1-N}{2} \|\Gamma\|_p \|v\|_{p'}.
\end{equation*}
  This yields for large enough balls $B$
\begin{equation*}
	\|\chi_{\R^N\setminus B} \mathcal R_k^\gamma[\Gamma v]\|_{q}
	\leq C \|\Gamma\|_{p} \|v\|_{p'} \Big(\int_{\R^N \setminus B} |x|^{\frac{q(1-N)}{2}}\dx \Big)^\frac{1}{q}
\end{equation*}
and the conclusion follows due to $q>\frac{2(N+1)}{N-1}>\frac{2N}{N-1}$.
\hfill $\square$

\subsection{Breather Solutions for the perturbed Wave Equation}

We consider
\begin{equation}\label{eq_perturbedwave}
	\partial_{tt}  U - \Delta U + V(x) U = Q(x) |U|^{p-2} U
	\qquad \text{on } \T \times \R^N
\end{equation}
where now $V$ is a short-range potential. In \cite[Theorem~4]{CosMan} 
the following generalization of Lemma~\ref{lem_LAP} is proved.

\begin{lem}[Cossetti, Mandel] \label{lem_LAP_CosMan}
  Let $N\in\N,N\geq 3$ and assume $V\in L^{\frac{N}{2}}(\R^N)+L^{\frac{N+1}{2}}(\R^N)$
  and $\frac{2(N+1)}{N-1}\leq r\leq \frac{2N}{N-2}$. For every $k\in\Z\sm\{0\}$ there is a
  bounded and symmetric distributional right  inverse
    $\mathcal{R}_k:L^{r'}(\R^N)\to L^r(\R^N)$ of $-\Delta+V(x)-k^2$ satisfying
  $$
    \|\mathcal R_k f\|_r \leq |k|^{-2+\frac{N}{r'}-\frac{N}{r}}\|f\|_{r'}.
  $$
\end{lem}

  \begin{cor} \label{cor:CossettiMandel}
    Assume $N\in\N,N\geq 2$, $V$ as in Lemma~\ref{lem_LAP_CosMan} and $Q\in L^\frac{q}{q-p}(\R^N),Q>0$
    where $p,q$ satisfy
\begin{align*}
	&2 < p < \frac{2(N+1)}{N-1},
	\qquad
	 \frac{2(N+1)}{N-1} < q < \frac{2Np}{(N-1)p - 2}.
\end{align*}
Then, for $s\in\{3,5\}$, the perturbed nonlinear  wave equation~\eqref{eq_perturbedwave} has an
unbounded sequence of $2\pi$-periodic real-valued  distributional breather solutions  $U_j \in
L^q(\R^N;L^p_s(\T)), \, j \in \N_0$.
  \end{cor}
  \begin{proof}
    In order to apply Theorem~\ref{thm_general}, we have to check the conditions (A1),(A2),(A3).
    Assumption (A1) follows from Lemma~\ref{lem_LAP_CosMan} and (A2) is a special case of
    \cite[Proposition~8]{CosMan}. To verify (A3) we proceed as outlined in the Introduction by choosing a test
    function  $\phi_k\in C_c^\infty(\R^N)$  such that $\int_{\R^N} \phi_k (-\Delta+V(x)-k^2)\phi_k\dx>0$.
     To this end we write $V=V_1 +V_2$ with $V_1\in L^{N/2}(\R^N),V_2\in
     L^{(N+1)/2}(\R^N)$. Replacing $V_1,V_2$ by
     $V_1 1_{|V|>R}$ respectively $V_1 1_{|V|\leq R}+V_2$ for large enough $R>0$ if necessary we can assume
     that $\|V_1\|_{N/2}\leq \delta$  some given $\delta>0$. From H\"older's inequality and
     the Interpolation inequality
     $$
       \|\psi\|_{\frac{2(N+1)}{N-1}}^2
       \leq \|\psi\|_{\frac{2N}{N-2}}^{\frac{2N}{N+1}} \|\psi\|_2^{\frac{2}{N+1}}
       \leq  \frac{\delta}{1+\|V_2\|_{\frac{N+1}{2}}}
       \|\psi\|_{\frac{2N}{N-2}}^2
        + C_\delta   \|\psi\|_2^2
        \qquad\text{for all }\psi\in C_c^\infty(\R^N)
     $$
     for some $C_\delta>0$ depending on $\delta,V_2$ we therefore get  
     \begin{align*}
       &\int_{\R^N} \phi_k (-\Delta+V(x)-k^2)\phi_k\dx \\
       &= \int_{\R^N} |\nabla\phi_k|^2 + (V(x)-k^2)|\phi_k|^2\dx \\
       &\geq \|\nabla\phi_k\|_2^2  -  \left( \|V_1\|_{\frac{N}{2}}
       \|\phi_k\|_{\frac{2N}{N-2}}^2 +
       \|V_2\|_{\frac{N+1}{2}}
       \|\phi_k\|_{\frac{2(N+1)}{N-1}}^2
       \right) - k^2 \|\phi_k\|_2^2 \\
       &\geq \|\nabla\phi_k\|_2^2  - 2\delta 
       \|\phi_k\|_{\frac{2N}{N-2}}^2 - (C_\delta+k^2)\|\phi_k\|_2^2  \\
       &\geq (1-2\delta  C_S^2) \|\nabla\phi_k\|_2^2  - (C_\delta+k^2)\|\phi_k\|_2^2.
   \end{align*}
   Here,  $C_S>0$ comes from Sobolev's Embedding Theorem. Choosing
   $0<\delta<\frac{1}{2C_S^2}$ and $\phi_k=\phi^*(t_k\cdot)$  for some fixed nontrivial $\phi^*\in
   C_c^\infty(\R^N)$ and large enough $t_k>0$ we get the result.
  \end{proof}

\subsection{Generalizations} \label{sec:gen}

Before going on with the proof of our main result we indicate further generalizations of
our method.
\begin{itemize}
  \item \textbf{(General periods)} In our main result we presented the theory for $2\pi$-periodic breathers.
  Clearly, by rescaling, there is an analogous theory for $T$-periodic breathers for any given $T>0$.
  Analytically this does not change much, but explicit criteria in applications
  need to be adapted. For instance, in Corollary~\ref{cor_kleingordon} dealing with the Klein-Gordon equation
  the conditions on the mass $m$ in (i),(ii),(iii) need to be replaced by the corresponding
  conditions on $\frac{Tm}{2\pi}$.
  \item \textbf{(Negative $Q$)} In (A2) we assume nonnegative $Q$ since this is required by the classical dual
  variational approach that we implement in our paper. In the context of Helmholtz-type problems it is
  possible to deal with $Q\leq 0$ as well. Indeed, following \cite[Section~3]{MaMoPe_Osc} we can
  slightly modify our functional $J$ from~\eqref{eq_wavefunctional} below to cover this case.
  In this way one obtains the existence of infinitely many breathers in that case.
  In the case of an elliptic operator $\L$, say $\L=-\Delta$, the corresponding solutions
  are not stationary and hence polychromatic provided that $Q<0$ (no matter what
  $s\in\{1,\ldots,5\}$ is).
  As explained right after Theorem~\ref{thm_general}, this indeed follows once we know that constant in time
  solutions do not exist, and that is true because the maximum principle
  implies that solutions of $-\Delta U = Q|U|^{p-2}U$ are necessarily trivial due to $Q< 0$.
  Sign-changing $Q$ can be treated in the context of   Helmholtz equations~\cite{ManSchYes} and it might be
  that these techniques can be adapted to construct breather solutions.
  \item \textbf{(Non-Euclidean Settings)}  Resolvent estimates of type $L^p-L^q$ also hold in the hyperbolic
  space, see~\cite[Theorem~2.3]{CasMan} and~\cite[Theorem~1.2]{Sogge}. The decay rate of the operator norms of the
  corresponding  distributional  right inverses with respect to $k$, however, is not known as far as we can
  see. We expect that our method applies once a bound as in (A1) is proved. 
  \item \textbf{(General evolutions)} The wave operator $\partial_{tt} + \L$ can be replaced by
  $P(-i\partial_t)+\L$ where $P:\R\to\R$ is a polynomial. In that case, distributional right 
  inverses for $\L+P(k)$ are needed instead of $\L-k^2$ and the definition of $\I_s$ and
  Proposition~\ref{prop_waveresolvent}~(i) have to be adapted in order to ensure the compatiblity of
  the operator $\mathscr R$ (see the following section) with the imposed symmetries. The results for odd
  respectively even polynomials $P$ will be different here.
  \item \textbf{(Systems)} One may ask whether breathers also exist for coupled nonlinear wave equations.
  Following our approach, this leads to infinite systems of coupled nonlinear Helmholtz
  systems. We believe that some ideas from the paper~\cite{ManSch_Dual} about $2\times 2$-Nonlinear
  Helmholtz Systems can be used.
  \item \textbf{(General nonlinearities)} Our paper deals with power-type nonlinearities, but the dual
  variational technique is actually  more flexible. To apply this method to a general nonlinearity
  $f(x,u)$ (replacing $Q(x)|u|^{p-2}u$ in~\eqref{eq_general}), one has to require the invertibility of
  $z\mapsto f(x,z)$ for almost all $x\in{\R^N}$. This is guaranteed by imposing a
  monotonicity assumption with respect to $z$. Moreover, this inverse needs to give rise to a dual functional having the
  Mountain Pass Geometry on an appropriate Banach space. In some cases, if the nonlinearity does not behave
  like a pure power, Orlicz spaces can be used, see for instance~\cite{Eveq_Orlicz}. Being
  interested in an unbounded sequence of breather solutions, one may further  impose that the
  nonlinearity is odd with respect to the second entry. We have to admit that the particularly important case
  $f(x,u)=|u|^{p-2}u$ is not covered by any of our examples. The reason is that assumption (A2) may not hold,
  for instance due to the translation invariance of $\mathcal R_k$. Here,   more sophisticated variational
  methods such as the Concentration-Compactness Principle could prove to be useful in order to overcome the
  lack of compactness. An idea for a fixed point approach aiming at the construction of small
  breather solutions for general nonlinearities or $f(x,u)= |u|^{p-2}u$ can be found in~\cite{Man_Uncountably}.
\end{itemize}

\section{Proof of Theorem~\ref{thm_general}} \label{sec:ProofThm}

To motivate our variational approach we introduce the formal Fourier series expansion
$$
  U(t,x)= \sum_{k\in \I_s} \e{\i kt}u_k(x)\quad
  \text{with Fourier modes }
	u_k(x) := [U]_k(x) := \frac{1}{2\pi} \int_\T \e{-\i k t} U(t,x) \dt.
$$
Recall that $\I_s$ collects the frequencies that are needed for building up breather solutions $U$ with the
symmetry indexed by $s\in\{1,\ldots,5\}$ as in the Introduction. Plugging this ansatz into~\eqref{eq_general}
we are lead to the infinite system of equations
\begin{align*}
	( \L - k^2)  u_k
	=  [Q|U|^{p-2}U]_k
	= Q^{1/p}\, [Q^{1/p'}|U|^{p-2}U]_k
	\qquad (k\in\I_s).
\end{align*}
We introduce the dual variable $V := Q^{1/p'}|U|^{p-2}U$ with formal Fourier series expansion
\begin{align*}
	V(t, x) = \sum_{k \in \I_s} \e{\i k t} v_k(x)
	\quad \text{with } v_k(x) = [V]_k(x) = \frac{1}{2\pi} \int_\T \e{-\i k t} V(t,x) \dt.
\end{align*}
To find solutions of the above-mentioned infinite system, it is sufficient
to solve the following   equations:
\begin{align*}
	&\qquad\quad Q^{1/p} u_k =  \mathcal{R}^Q_k \left[  [Q^{1/p'}|U|^{p-2}U]_k \right]\qquad\text{for all
	}k\in\I_s \\
	& \rightsquigarrow \quad
	Q^{1/p} U(t,\cdot) = \sum_{k \in \I_s} \e{\i k t}  \mathcal{R}^Q_k \left[  [Q^{1/p'}|U|^{p-2}U]_k \right]
	\\
	& \rightsquigarrow \quad
	|V(t,\cdot)|^{p'-2}V(t,\cdot)  = \sum_{k \in \I_s} \e{\i k t}  \mathcal{R}^Q_k \left[  v_k
	\right]	\\
 	& \rightsquigarrow \quad
 	|V|^{p'-2}V  = \mathscr{R}[V]
\end{align*}
where the operator $\mathscr{R}$ is defined via
\begin{equation*}
	\mathscr{R}[V] (t, x) :=
	\sum_{k \in \I_s} \e{\i k t}  \mathcal{R}^Q_k \left[  v_k \right](x)
	\quad \text{with }
	v_k(x) = [V]_k(x) = \frac{1}{2\pi} \int_\T \e{-\i k t} V(t,x) \dt.
\end{equation*}
Similarly, we can  derive the formula $U = (Q^{-1/p}\mathscr R)[V]$ where
$$
  (Q^{-1/p}\mathscr R)[V]  := \sum_{k\in\I_s} \e{\i kt} \mathcal R_k\big[Q^{1/p}v_k\big],
$$
Notice that this operator makes sense under our hypothesis $Q\geq 0$ and
$Q^{1/p}(Q^{-1/p}\mathscr R)[V]= \mathscr R[V]$. 
Since $\mathscr{R}$ will turn out to be symmetric, the above equation for $V$ has a variational structure. It
is the Euler-Lagrange equation of the functional
\begin{align}\label{eq_wavefunctional}
	J(V) := \frac{1}{p'} \int_{\T \times {\R^N}} |V|^{p'} \dtx
	- \frac{1}{2} \int_{\T \times {\R^N}} V	\mathscr{R}[V]	\dtx
\end{align}
so that we are lead to prove the existence of critical points. This motivates the following discussion of
the functional $J$ and finishes the nonrigorous introductory part of this section.

\medskip

We start our rigorous analysis of the functional by proving that  $J$ is well-defined and continuously
differentiable on the Banach space $X^{p'}_s$ where, from now on, $X^r_s:= L^r({\R^N}, L^r_s(\T))$ for $r\in
(1,\infty)$ and $s\in\{1,..,5\}$. These spaces were introduced at the beginning of the paper.
We will
need the Hausdorff-Young inequality for Fourier series that we recall for the convenience of the reader.

\begin{prop}[Hausdorff-Young] \label{prop_HY}
  Let $p\in [2,\infty]$. Then there is a $C>0$ such that
 \begin{align}
  \|f\|_{L^{p}(\T)}&\leq C \|\hat f \|_{\ell^{p'}(\Z)} \label{eq:HYI} \\
  \|\hat g\|_{\ell^p(\Z)} &\leq C \|g\|_{L^{p'}(\T)} \label{eq:HYII}
\end{align}
whenever $\hat f\in \ell^{p'}(\Z)$ and $g\in L^{p'}(\T)$. Here, $\hat g(k):= \frac{1}{2\pi} \int_\T \e{-\i k
t} g(t) \dt$ for $k\in\Z$.
\end{prop}

The proofs of the following propositions are postponed to Section~\ref{sec:ProofAux}.

\begin{prop}\label{prop_waveresolvent}
Assume~(A1),(A2).
\begin{itemize}
  \item[(i)] The operator $\mathscr{R}: X^{p'}_s\to X^p_s$ is well-defined, continuous, symmetric and compact.
  \item[(ii)] The operator $Q^{-1/p}\mathscr{R}: X^{p'}_s\to L^q({\R^N},L^p_s(\T))$ is  well-defined and
  continuous.
\end{itemize}
\end{prop}

Using part~(i) of the previous proposition we show that $J$ satisfies the
assumptions of the Symmetric Mountain Pass Theorem. This   allows to conclude that there is an unbounded
sequence of critical points. Those will provide the breather solutions $U$ after inverting the formal passage
to the dual variables from the beginning of this section.

\begin{prop}\label{prop_wavefunctional}
Assume~(A1),(A2),(A3). Then the functional $J: \:X^{p'}_s \to \R$ as
in~\eqref{eq_wavefunctional} is even, continuously differentiable and has the Mountain Pass Geometry:
\begin{itemize}
\item[(i)] $J(0) = 0$ and there are $r,\delta > 0$ with $J(V) \geq \delta$ for all $V \in X^{p'}_s$ with
$\norm{V}_{p'} = r$.
\item[(ii)] There is an increasing sequence of linear subspaces $\mathfrak{W}^{(m)} \subseteq X^{p'}_s$ of
dimension $m$ and radii $R_m > r$ such that $J(V) < 0$ for all $V \in \mathfrak{W}^{(m)}$ with $\norm{V}_{p'} > R_m$.
\item[(iii)] $J$ satisfies the Palais-Smale condition.
\end{itemize}
\end{prop}

The last of our preparatory results shows that each critical point of $J$ indeed provides a $2\pi$-periodic
real-valued  distributional breather solution as claimed in Theorem~\ref{thm_general}. Here we use
part~(ii) of Proposition~\ref{prop_waveresolvent}.

\begin{prop}\label{prop_wavesolution}
  Assume~(A1),(A2) and let  $V \in X^{p'}_s$ be a nontrivial critical point of $J$. Then the
  function $U := (Q^{-1/p}\mathscr R)[V]\in L^q({\R^N},L^p_s(\T))$ is a nontrivial
  $2\pi$-periodic real-valued  distributional breather solution of the nonlinear wave
  equation~\eqref{eq_general} in the sense of~\eqref{eq_weaksolution}.
\end{prop}

We summarize the arguments mentioned above to prove our main result.

\medskip

\noindent\textbf{Proof of Theorem~\ref{thm_general}:} \\
By Proposition~\ref{prop_wavefunctional} the functional $J$ satisfies all assumptions of the Symmetric
Mountain Pass Theorem~\cite[Corollary~7.23]{ghoussoub}. So there is an unbounded sequence of
critical values of $J$. Since $J$ maps bounded sets to bounded sets, we thus get an unbounded sequence of
critical points $(V_j)_{j\in\N_0}$.
By Proposition~\ref{prop_wavesolution}, the substitution $U_j := (Q^{-1/p}\mathscr R)[V_j]$
yields the asserted infinite sequence of distributional $2\pi$-periodic breather solutions of the nonlinear wave
equation~\eqref{eq_general} in $L^q({\R^N},L^p_s(\T))$. This sequence is unbounded because
H\"older's inequality implies
\begin{align*}
    \|Q\|_{\frac{q}{q-p}}^{1/p} \|U_j\|_q
  =\|Q^{1/p}\|_{\frac{pq}{q-p}} \|U_j\|_q
  \geq \|Q^{1/p} U_j\|_p
  = \|V_j\|_{p'}^{p'-1}
  \nearrow \infty \quad (j\to\infty).
\end{align*}
Here, in the last equality we used $Q^{1/p} U_j = Q^{1/p} (Q^{-1/p}\mathscr R)[V_j] = \mathscr
R[V_j]=|V_j|^{p'-2}V_j$ where the last equality follows from $J'(V_j)=0$.  \hfill $\square$

\section{Proofs of auxiliary results}  \label{sec:ProofAux}  

\subsection*{Proof of Proposition~\ref{prop_waveresolvent}} ~

\begin{steps}

\step{Proof of (i) -- Well-definedness and continuity.}
We show that the series in the definition of $\mathscr{R}[V]$ converges in $L^p({\R^N},L^p(\T))$ and that
$\mathscr{R}$ preserves the time-symmetry, i.e., $\mathscr{R}(X^{p'}_s)\subset X^p_s$.
To prove the first point we use the estimate
\begin{align} \label{eq:Bound:RQ}
  \begin{aligned}
  \|\mathcal{R}^Q_k w\|_p
  &= \|Q^{1/p}\mathcal{R}_k (Q^{1/p}w)\|_p \\
  &\leq \|Q^{1/p}\|_{\frac{pq}{q-p}}  \|\mathcal{R}_k (Q^{1/p}w)\|_q \\
  &\stackrel{(A1)}\leq \|Q\|_{ \frac{q}{q-p} }^{1/p} C(k^2+1)^{-\alpha/2} \|Q^{1/p}w\|_{q'}
  \\
  &\leq \|Q\|_{ \frac{q}{q-p} }^{1/p}  \|Q^{1/p}\|_{ \frac{p'q'}{p'-q'}}
  C(k^2+1)^{-\alpha/2}  \|w\|_{p'} \\
  &= \|Q\|_{\frac{q}{q-p}}^{2/p}
  C(k^2+1)^{-\alpha/2}  \|w\|_{p'} \\
  &\leq C_1(k^2+1)^{-\alpha/2} \|w\|_{p'}
  \qquad\text{where }C_1:= C\|Q\|_{\frac{q}{q-p}}^{2/p}.
\end{aligned}
\end{align}
Next we prove an estimate for sums over finitely many modes $k\in\I_s$ that will imply the
well-definedness of $\mathscr R$ after some Cauchy sequence argument. So let $\J_s\subset \I_s$ be any finite
subset. Then:
\begin{align*}
	 \norm{\sum_{k\in\J_s}
	\e{\i k\, \cdot \, } \: \mathcal{R}^Q_k\left[ v_k \right]}_{L^p(\R^N,L^p(\T))}
	& =
	\norm{ \norm{
	\sum_{k\in\J_s}
	\e{\i k\, \cdot \, } \:\mathcal{R}^Q_k\left[ v_k \right]
	}_{L^p(\T)} }_p
	\\
	&  \stackrel{\eqref{eq:HYI}}\leq
	\norm{
	\left(
	\sum_{k\in\J_s}
	| \mathcal{R}^Q_k\left[ v_k \right] |^{p'}
	\right)^{1/p'}
	}_p  \\
	&=
	\norm{
	\sum_{k\in\J_s}
	| \mathcal{R}^Q_k\left[ v_k \right] |^{p'}	}_{p-1}^{1/p'}	  \\
	&\stackrel{p>2}\leq
	\left(\sum_{k\in\J_s}
	\norm{| \mathcal{R}^Q_k\left[ v_k \right] |^{p'}
	}_{p-1} \right)^{1/p'}	  \\
	&= \left(\sum_{k\in\J_s}
	\norm{
	\mathcal{R}^Q_k\left[ v_k \right]}_p^{p'}
	 \right)^{1/p'}
	 \\
	 &\stackrel{\eqref{eq:Bound:RQ}}\leq
	C_1  \:
	 \left( \sum_{k\in\J_s}
	(k^2 + 1)^{- \frac{\alpha p'}{2}}
	\norm{ v_k 	}^{p'}_{p'}
	 \right)^{1/p'}
		  \\
	 &= C_1  \:
	 \left( \int_{\R^N} \sum_{k\in\J_s}
	(k^2 + 1)^{- \frac{\alpha p'}{2}}
	|v_k(x)|^{p'} \dx \right)^{1/p'} \\
	&\stackrel{p>2}\leq C_1  \:
	 \left( \int_{\R^N}  \left(
	 \sum_{k\in\J_s} (k^2 + 1)^{- \frac{\alpha p'}{2} \cdot \frac{p-1}{p-2}}
	 \right)^{\frac{p-2}{p-1}} \left(\sum_{k\in\J_s}
	|v_k(x)|^{p}\right)^{\frac{p'}{p}} \dx \right)^{1/p'} \\
	&\stackrel{\eqref{eq:HYII}}\leq
	 C_1 \: \left(
	 \sum_{k\in\J_s} (k^2 + 1)^{- \frac{\alpha p}{2(p-2)}}
	 \right)^{\frac{p-2}{p}} \:
	 \norm{  \|V(\cdot,x)\|_{L^{p'}(\T)}  }_{p'}  \\
	&=	 C_1 \:\left(
	 \sum_{k\in\J_s}  (k^2 + 1)^{- \frac{\alpha p}{2(p-2)}}
	 \right)^{\frac{p-2}{p}}\: \norm{V}_{L^{p'}(\R^N,L^{p'}(\T))}.
\end{align*}
So we have proved 
\begin{equation}\label{eq_cauchyestimate}
	 \norm{\sum_{k\in\J_s}}
	\e{\i k\, \cdot \, } \: \mathcal{R}^Q_k\left[ v_k \right]_{L^p(\R^N,L^p(\T))}
	\leq
	C_1 \:
	 \left(
	 \sum_{k\in\J_s} (k^2 + 1)^{- \frac{\alpha p}{2(p-2)}}
	 \right)^{\frac{p-2}{p}}\norm{V}_{L^{p'}(\R^N,L^{p'}(\T))}.
\end{equation}
In view of \eqref{eq_cauchyestimate} and $\frac{\alpha p}{p-2} > 1$, which holds by assumption (A1),
a Cauchy sequence argument proves that the corresponding infinite sum converges. So $\mathscr{R}
[V]\in L^p({\R^N},L^p(\T))$  is well-defined  and 
\begin{equation*}
  \norm{\mathscr{R}[V]}_{L^p(\R^N,L^p(\T))} \leq C_1 \:
  \left(\sum_{k\in\I_s} (k^2 + 1)^{- \frac{\alpha p}{2(p-2)}} \right)^{\frac{p-2}{p}}
  \norm{V}_{L^{p'}(\R^N,L^{p'}(\T))}.
\end{equation*}
Notice that $\mathscr{R}$ maps real-valued functions to real-valued functions because so do the operators
$\mathcal{R}^Q_k$. So it remains to prove the time-symmetry preserving property of $\mathscr R$, i.e.,
$\mathscr{R}(X^{p'}_s)\subset X^p_s$. Recall that elements $V\in X^{p'}_s$ are real-valued by assumption and
thus satisfy $v_k=\overline{v_{-k}}$ for all $k\in \Z$ and all $s=1,\ldots,5$.
\begin{itemize}
  \item For $s=1$ there is nothing to prove.
  \item For $s=2$ any $V\in X^{p'}_s$ is even in time. Equivalently, all Fourier modes $v_k=v_{-k}$ are
  real-valued.
  This is true also for $\mathscr R V$ because $\mathcal R_k=\mathcal R_{-k}$ by (A1)  implies
  \begin{align*}
    [\mathscr R V]_k &= \mathcal R_k^Q[v_k] = \mathcal R_{-k}^Q[v_{-k}] = [\mathscr R V]_{-k},\\
    [\mathscr R V]_k
    &= \mathcal R_k^Q[v_k]
    = \mathcal R_k^Q[\overline{v_k}]
    = \overline{\mathcal R_k^Q[ v_k]}
    = \overline{ [\mathscr R V]_k }.
  \end{align*}
  \item For $s=3$ any $V\in X^{p'}_s$ is odd in time, i.e., the zero mode $v_0=0$ vanishes and
  the other Fourier modes $v_k=-v_{-k}$ are purely imaginary. Again,
  this is true also for $\mathscr R V$ because the zero mode does not occur in $\I_s$ and
  \begin{align*}
    [\mathscr R V]_k &= \mathcal R_k^Q[v_k] = \mathcal R_{-k}^Q[-v_{-k}] = -[\mathscr R V]_{-k},\\
    [\mathscr R V]_k
    &= \mathcal R_k^Q[v_k]
    = \mathcal R_k^Q[-\overline{v_k}]
    = -\overline{ [\mathscr R V]_k }.
  \end{align*}
  \item For $s=4$ any $V\in X^{p'}_s$ is $\pi$-periodic, i.e., the modes with odd $k$ vanish. Since $\I_4=2\Z$
  this is true as well for $\mathscr R V$.
  \item For $s=5$ any $V\in X^{p'}_s$ is $\pi$-antiperiodic, i.e., the modes with even $k$ vanish. Since
  $\I_5=2\Z+1$ this is true as well for $\mathscr R V$.
\end{itemize}

\step{Proof of (i) -- Symmetry and compactness.}

The symmetry of $\mathscr{R}$, which means
\begin{align*}
	\int_{\T \times {\R^N}} \mathscr{R}[V](t,x) W(t,x) \: \dtx
	= \int_{\T \times {\R^N}} V(t,x) \mathscr{R}[W](t,x) \: \dtx
\end{align*}
for all $V, W \in X^{p'}_s$, follows from the the continuity of $\mathscr R:X_s^{p'}\to X_s^p$ and
the symmetry of $\mathcal{R}_k^Q$.
We now turn to the proof of compactness. Assume that $(V^{(n)})_n$ is a bounded
sequence in $X^{p'}_s$ with $\norm{V^{(n)}}_{p'} \leq C_V$ for all $n \in \N$. We aim to show that a
subsequence of $(\mathscr{R}[V^{(n)}])_n$ converges in $X^{p'}_s$. Here, for almost all $t \in \T$ and $n \in
\N$,
\begin{align*}
	\mathscr{R} [V^{(n)}](t, \, \cdot \,)
	= \sum_{k \in \I_s} \e{\i kt} \: \mathcal{R}^Q_k [v_k^{(n)}]
	\qquad
	\text{in } L^p({\R^N}).
\end{align*}
From H\"older's inequality we infer
\begin{align*}
	\norm{v_k^{(n)}}_{p'}
	= \norm{ \frac{1}{2\pi} \: \int_{\T} \e{\i kt} \: V^{(n)}(t, \, \cdot \, ) \dt }_{p'}
	\leq (2\pi)^{-\frac{1}{p'}} \: \norm{V^{(n)}}_{L^{p'}(\R^N,L^{p'}(\T))}  \leq C_V.
\end{align*}
So all sequences $(v_k^{(n)})_n$ are bounded in $L^{p'}({\R^N})$. The compactness of $\mathcal R_k^Q$ from
(A2) combined with a standard diagonal sequence argument provides $y_k \in L^p({\R^N}), k\in\I_s,$ and
a subsequence $(v_k^{(n_j)})_j$ with
\begin{align}\label{eq_diagonalsequence}
	\forall \: k \in \I_s
	\qquad
	\mathcal{R}^Q_k [v_k^{(n_j)}] \to y_k
	\qquad
	\text{ in } L^p({\R^N}) \text{ as } j \to \infty .
\end{align}
We claim that
this implies
\begin{equation}\label{eq:PS}
  \mathscr{R}  [V^{(n_j)}]\to  \sum_{k \in \I_s} \e{\i k \,\cdot \,} y_k \qquad\text{in
  }X^p_s\;\text{as }j\to\infty.
\end{equation}

\medskip

Before verifying~\eqref{eq:PS} we check that the term on the right indeed belongs to $L^p(\R^N,L^p(\T))$.
Indeed, for all finite $\J_s\subset\I_s$ we have
\begin{align*}
	\norm{ \sum_{k\in\J_s} \e{\i k\cdot} \, y_k }_{L^p(\R^N,L^p(\T))}
	& =
	\lim_{j \to \infty} \norm{ \sum_{k\in\J_s} \e{\i k\cdot} \: \mathcal{R}^Q_k [v_k^{(n_j)}]
	}_{L^p(\R^N,L^p(\T))}
	\\
	&\stackrel{\eqref{eq_cauchyestimate}}\leq
	C_1 \left( \sum_{k\in\J_s} (k^2 + 1)^{-\frac{\alpha p}{2(p-2)}} \right)^\frac{p-2}{p}  \limsup_{j \to \infty}
	\norm{V^{(n_j)}}_{L^{p'}(\R^N,L^{p'}(\T))}
	\\
	& \leq
	C_1 C_V \left( \sum_{k\in\J_s} (k^2 + 1)^{-\frac{\alpha p}{2(p-2)}} \right)^\frac{p-2}{p}
	<\infty.
\end{align*}
As above, this and $\alpha>1-\frac{2}{p}$ implies $\sum_{k \in \I_s} \e{\i k \,\cdot \,}  y_k\in
L^p(\R^N,L^p(\T))$.
Next we verify~\eqref{eq:PS}. So let $\eps>0$. From the previous statement we obtain some finite subset
$\J_s\subset\I_s$ such that for all $j\in\N$
\begin{align*}
	\norm{ \sum_{k\in\I_s\sm \J_s} \e{\i k\cdot}  y_k }_{L^p(\R^N,L^p(\T))}
	&< \frac{\eps}{4},
	\\
	\norm{ \sum_{k\in\I_s\sm\J_s} \e{\i k\cdot}
	\mathcal{R}^Q_k [v_k^{(n_j)}] }_{L^p(\R^N,L^p(\T))}
	&\overset{\eqref{eq_cauchyestimate}}{\leq} C_R \: \left( \sum_{k\in\I_s\sm\J_s}
	(k^2 + 1)^{- \frac{\alpha p}{2(p-2)}}
	\right)^{\frac{p-2}{p}}
	 \norm{V^{(n_j)}}_{L^{p'}(\R^N,L^{p'}(\T))} 	\\
	&\: \leq C_R C_V \: \left( \sum_{k\in\I_s\sm\J_s}
	(k^2 + 1)^{- \frac{\alpha p}{2(p-2)}}
	\right)^{\frac{p-2}{p}} \\
	 &< \frac{\eps}{4}.
\end{align*}
So~\eqref{eq_diagonalsequence} and again the Hausdorff-Young
inequality yield some $j_0 = j_0(\eps) \in \N$ such that for all $j\geq j_0$ the following holds.
\begin{align*}
	\norm{ \sum_{k\in\J_s} \e{\i k\cdot} \, \mathcal{R}^Q_k [v_k^{(n_j)}]
	- \sum_{k\in\J_s} \e{\i k\cdot} \, y_k}_{L^p(\R^N,L^p(\T))}
	\leq \left(
	\sum_{k\in\J_s}
	\norm{ \mathcal{R}^Q_{k} [v_k^{(n_j)}] - y_k}_{p}^{p'}
	\right)^\frac{1}{p'}
	\!\!\! < \frac{\varepsilon}{2}.
\end{align*}
Combining all these estimates, we infer for $j \geq j_0$
\begin{align*}
	  \norm{ \sum_{k \in \I_s} \e{\i k\cdot} \,  \mathcal{R}^Q_{k} [v_k^{(n_j)}]
	- \sum_{k \in \I_s} \e{\i k\cdot} \,   y_k}_{L^p(\R^N,L^p(\T))}
	< \varepsilon,
\end{align*}
which finishes the proof of~\eqref{eq:PS} because $X^p_s$ is closed in $L^p(\R^N,L^p(\T))$.

\step{Proof of (ii).}

 We essentially repeat the estimate from the first step where the exponent $q$ replaces
 $p$ in the spatial Lebesgue norm in order to take the missing factor $Q^{1/p}$ into account.
 Given that $q\geq p'$ and that our summation is over finitely many indices $k\in\J_s$  only, we obtain
\begin{align*}
   \norm{
	\norm{  \sum_{k\in\J_s} \e{\i k \cdot} \mathcal{R}_{k}[Q^{1/p} v_k]}_{L^p(\T)}}_q
	&\stackrel{ q\geq p'}\leq \left(\sum_{k\in\J_s}
	\norm{  \mathcal{R}_{k}[Q^{1/p} v_k]}_q^{p'}
	 \right)^{1/p'}
	 \\
	 &\leq	C     \:
	 \left( \sum_{k\in\J_s}
	(k^2 + 1)^{- \frac{\alpha p'}{2}}
	\norm{Q^{1/p} v_k 	}^{p'}_{q'}
	 \right)^{1/p'}
		  \\
   &\leq	C     \:
	 \left( \sum_{k\in\J_s}
	(k^2 + 1)^{- \frac{\alpha p'}{2}}
	\|Q^{1/p}\|_{\frac{p'q'}{p'-q'}}^{p'} \norm{ v_k 	}^{p'}_{p'}
	 \right)^{1/p'}
		  \\
	&\leq	C\|Q\|_{\frac{q}{q-p}}^{1/p}     \:
	 \left( \sum_{k\in\J_s}
	(k^2 + 1)^{- \frac{\alpha p'}{2}} \norm{ v_k 	}^{p'}_{p'}
	 \right)^{1/p'}
		  \\
	&\stackrel{\text{step 1}}\leq C\|Q\|_{\frac{q}{q-p}}^{1/p}  \:\left(
	 \sum_{k\in\J_s} (k^2 + 1)^{- \frac{\alpha p}{2(p-2)}}
	 \right)^{\frac{p-2}{p}}\: \|V\|_{L^{p'}(\R^N,L^{p'}(\T))}.
\end{align*}
Since the sum on the right is bounded independently of   $\J_s\subset\I_s$, we get the result.

\end{steps}

\hfill $\square$

\medskip

\subsection*{Proof of Proposition~\ref{prop_wavefunctional}} ~

We prove that the functional
\begin{align*}
	J: X_s^{p'}  \to \R,
	\quad
	J(V) := \frac{1}{p'}\int_{\T \times {\R^N}} |V|^{p'} \dtx - \frac{1}{2} \int_{\T \times {\R^N}} V
	\mathscr{R}[V] \dtx
\end{align*}
satisfies the assumptions of the Symmetric Mountain Pass Theorem. It is straightforward to deduce from
Proposition~\ref{prop_waveresolvent} and $X_s^{p'}\subset L^{p'}(\R^N,L^{p'}(\T))$ 
that $J$ is well-defined, even and of class $C^1$.

\medskip

\textbf{(i)}
Assuming $\norm{V}_{L^{p'}(\R^N,L^{p'}(\T))} = r$, we estimate using $C_R := \norm{\mathscr{R}}_{p' \to p} < \infty$ and get
\begin{align*}
	J(V)
	&= \frac{1}{p'} \int_{\T \times {\R^N}} |V|^{p'} \dtx - \frac{1}{2} \int_{\T \times {\R^N}} V \mathscr{R}[V]
	\dtx
	\\
	&\geq \frac{1}{p'} \norm{V}_{L^{p'}(\R^N,L^{p'}(\T))}^{p'} - \frac{C_R}{2} \norm{V}^2_{L^{p'}(\R^N,L^{p'}(\T))}  \\
	&= r^{p'} \left( \frac{1}{p'} 	- \frac{C_R}{2}  r^{2-p'} \right).
\end{align*}
Hence, the claim~(i) holds for $r = (C_R p')^{-1/(2 - p')}$ and $\delta =  r^{p'} / 2p'>0$.

\medskip

\textbf{(ii)}
According to (A3) we find $\omega_k \in L^{p'}({\R^N})$ such that  w.l.o.g.
\begin{align*}
	\int_{{\R^N}} \omega_k \mathcal{R}^Q_k[\omega_k] \dx = \frac{2}{\pi}
	\quad \text{for all } k \in \I_s.
\end{align*}
 With that, we choose  for positive  $k\in\I_s$
\begin{align*}
	V_k(t, x)
	:= w_k(x) T_k(t)
	:= \begin{cases}
	 w_k(x) \cos(kt) &\text{if }s\in\{1,2,4\}, \\
	 w_k(x) \sin(kt) &\text{if }s\in\{3,5\}.
	\end{cases}
\end{align*}
This choice guarantees $V_k\in X^{p'}_s$ for all $k\in\I_s$. Moreover,  for positive $k,k'\in\I_s$ 
\begin{align*}
	&\int_{\T \times {\R^N}} V_{k'} \: \mathscr{R} [V_k] \dtx
	=
	\int_{\T}  T_{k'}(t) T_k(t)	\dt
	\cdot
	\int_{{\R^N}} \omega_{k'} \, \mathcal{R}^Q_k[\omega_k] \dx
	\begin{cases}
	= 0 & \text{ if } k \neq k',
	\\
	= 2 & \text{ if } k = k'.
	\end{cases}
\end{align*}
So the $V_k$ are linearly independent. Indeed, $\sum_{k\in\J_s} c_k V_k=0$ for some
$c_k\in\R$ and finite subset $\J_s\subset\I_s$ implies $c_{k'}=0$ for all $k'\in\J_s$ because of
$$
  0
  = \int_{\T \times {\R^N}} V_{k'} \: \mathscr{R} \left[\sum_{k\in\J_s} c_k V_k\right] \dtx
  = \sum_{k\in\J_s} c_k  \int_{\T \times {\R^N}} V_{k'} \: \mathscr{R} [V_k] \dtx
  =  2 c_{k'}.
$$
 Choosing nested subsets $\I_s^j\subset \I_s$ with $j$ positive  elements
and $\mathfrak{W}_j:=\spa\{ V_k:k\in\I_s^j\}$ we thus get
$\text{dim } \mathfrak{W}_j = j$. For any fixed $j \in \N$, equivalence of norms   provides a constant $c_j > 1$ with
\begin{align*}
	\frac{1}{c_j} \left( \sum_{k\in \I_s^j}  \beta_k^2 \right)^{1/2} \leq
	\norm{\sum_{k\in \I_s^j} \beta_k V_k}_{L^{p'}(\R^N,L^{p'}(\T))} \leq
	c_j \left( \sum_{k\in \I_s^j}  \beta_k^2 \right)^{1/2}
	\:\: \text{whenever } \beta_k\in \R, k\in \I_s^j.
\end{align*}
For $R > r$ and some arbitrary element $V = \sum_{k\in \I_s^j} \beta_k V_k \in \mathfrak{W}_j$ with
$\norm{V}_{L^{p'}(\R^N,L^{p'}(\T))} = R$, we obtain the estimate
\begin{align*}
 	J(V)	&= \frac{1}{p'}
	\int_{\T \times {\R^N}} |V|^{p'} \dtx
	 - \frac{1}{2} \int_{\T \times {\R^N}} V \: \mathscr{R}[V] \dtx
	 \\
	 &=
	 \frac{1}{p'} \int_{\T \times {\R^N}} |V|^{p'} \dtx  - \frac{1}{2} \sum_{k,k'\in \I_s^j} \beta_k
	 \beta_{k'} \int_{\T \times {\R^N}} V_{k'} \: \mathscr{R}[V_k] \dtx
	  \\
	 &=   \frac{1}{p'} \cdot  R^{p'}  - \sum_{k\in \I_s^j}  \beta_k^2 \\
 	&\leq
 	 \frac{1}{p'} \cdot R^{p'}  - \frac{1}{c_j^2} \cdot R^2.
\end{align*}
Since $p' < 2$, we thus conclude for $R_j := \max \left\{ r, \left(c_j^2/p'\right)^{1/(2-p')} \right\}$ that
$J(V) < 0$ whenever $V \in \mathfrak{W}_j$ with $\norm{V}_{L^{p'}(\R^N,L^{p'}(\T))} > R_j$.

\medskip

\textbf{(iii)}
Take any Palais-Smale sequence $(V_n)_n$ for $J$, that is, $V_n \in X^{p'}_s$ with
\begin{align*}
	J'(V_n) \to 0 \quad \text{in } \left(X^{p'}_s\right)' = X^p_s,
	\qquad
	J(V_n) \to c
\end{align*}
where $c > 0$ denotes the Mountain Pass level. We claim that the sequence $(V_n)_n$ is bounded. Indeed,
assuming otherwise, the identity
\begin{align*}
	J'(V_n)[V_n] - 2 J(V_n)
	=  \left(1-\frac{2}{p'}\right) \int_{\T \times {\R^N}} |V_n|^{p'} \dtx
\end{align*}
leads in the limit $n \to \infty$ to the contradictory statement
\begin{align*}
	0
	= \limsup_{n\to\infty} \frac{J'(V_n)[V_n]-2J(V_n)}{\norm{V_n}_{L^{p'}(\R^N,L^{p'}(\T))}}
	= \limsup_{n\to\infty}  \left(1-\frac{2}{p'}\right)  \norm{V_n}_{L^{p'}(\R^N,L^{p'}(\T))}^{p'-1}
	= - \infty.
\end{align*}
Hence, we may assume w.l.o.g. that $V_n \rightharpoonup V$ weakly in $X^{p'}_s$ for some $V\in X^{p'}_s$.
Due to the compactness of $\mathscr{R}$ (see Proposition~\ref{prop_waveresolvent}),
this implies $\mathscr{R} [V_n] \to \mathscr{R}[V]$ strongly in $L^p_s(\T \times {\R^N})$. Hence, we obtain
\begin{align*}
	\int_{\T \times {\R^N}} V_n \mathscr{R}[V_n] \dtx
	\to
	\int_{\T \times {\R^N}} V \mathscr{R}[V] \dtx .
\end{align*}

As in the proof of~\cite[Lemma~5.2]{EvequozWeth} we conclude $V_n \to V$ in the strong sense. Indeed, weak
convergence implies
$$
  \|V\|_{L^{p'}(\R^N,L^{p'}(\T))}\leq \liminf_{n\to\infty} \|V_n\|_{L^{p'}(\R^N,L^{p'}(\T))}
$$
and the convexity of $t\mapsto |t|^{p'}$ yields, as $n\to\infty$,
\begin{align*}
  &\frac{1}{p'}\|V\|_{L^{p'}(\R^N,L^{p'}(\T))} - \frac{1}{p'} \|V_n\|_{L^{p'}(\R^N,L^{p'}(\T))} \\
  &\geq  \int_{\T \times \R^N} |V_n|^{p'-2}V_n(V-V_n)\dtx \\
  &=  J'(V_n)[V-V_n] + \int_{\R^N} V_n \mathscr R[V-V_n] \dtx = o(1),
\end{align*}
whence
$$
  \limsup_{n\to\infty} \|V_n\|_{L^{p'}(\R^N,L^{p'}(\T))} \leq \|V\|_{L^{p'}(\R^N,L^{p'}(\T))}.
$$
We conclude that the sequence of norms converges to the norm of the weak limit $V$. By uniform convexity of
$L^{p'}(\T \times \R^N)$, this implies $V_n\to V$ as $n\to\infty$ and the statement is proved.
\hfill $\square$

\medskip

\subsection*{Proof of Proposition~\ref{prop_wavesolution}} ~

We consider a nontrivial critical point $V\in L^{p'}_s(\T \times {\R^N}) $ of the functional $J$. Since
$\mathscr R$ is symmetric by Proposition~\ref{prop_waveresolvent}~(i), the Euler-Lagrange equation reads
\begin{align}\label{eq_criticalpoint}
	|V|^{p'-2}V = \mathscr{R} [V] \qquad \text{in } X^p_s = L^p({\R^N},L^p_s(\T)).
\end{align}
From  Proposition~\ref{prop_waveresolvent}~(ii) and~\eqref{eq_criticalpoint} we infer $U :=
(Q^{-1/p}\mathscr R)[V]\in L^q({\R^N},L^p_s(\T))$.  We will use
$Q^{1/p} U = Q^{1/p}\cdot  (Q^{-1/p}\mathscr R)[V]  = \mathscr R[V] = |V|^{p'-2}V$ and thus
$$
  Q|U|^{p-2}U
  = Q^{1/p}\cdot |Q^{1/p} U|^{p-2}Q^{1/p} U
  = Q^{1/p} V.
$$
 Using these facts, we have  to  verify
$$
  \int_{\T \times {\R^N}} Q|U|^{p-2}U \:\Phi \dtx
	= \int_{\T \times {\R^N}} U \: (\partial_{tt} + \L )  \Phi \dtx
$$
for all $C_c^\infty({\R^N},C^\infty(\T))$. To this end we proceed step by step. 
\medskip

We first verify the above identity for real-valued test functions of the form
$$
  \Phi(t, x) :=  \sum_{k\in\J_s} \e{-\i k t}   \phi_k(x),\qquad
  \J_s\subset\I_s \text{ finite },
  \phi_k\in C_c^\infty({\R^N})\,(k\in\J_s) \text{ such that }\Phi\in X^{p'}_s.
$$
To see this we use the Euler-Lagrange equation and the definition of $U$ from above.
\begin{align*}
	&\int_{\T \times {\R^N}} Q(x)|U(t,x)|^{p-2}U(t,x) \Phi(t,x)  \dtx
	\\
	& \quad =
	\int_{\T \times {\R^N}} Q(x)^{1/p}V(t,x) \Phi(t,x)  \dtx
	\\
	&\quad =
	\sum_{k\in\J_s} \int_{{\R^N}} Q(x)^{1/p} \phi_k(x)
	\left[ \int_{\T}  \e{-\i k t} \, V(t, x) \dt \right]
	 \dx
	\\
	&\quad = 2\pi
	\sum_{k\in\J_s} \int_{{\R^N}}  Q(x)^{1/p}v_k(x)\: \phi_k(x) \dx\\
	&\quad \overset{(A1)}{=} 2\pi
	\sum_{k\in\J_s} \int_{{\R^N}} \mathcal{R}_k [Q^{1/p}  v_k] (x)\:
	 (\L - k^2) \phi_k(x)
	\dx\\
	&\quad =  2\pi
	 \int_{{\R^N}} \sum_{k\in\J_s}
	 \mathcal{R}[Q_k^{1/p} v_k] (x)	 \:
	(\L - k^2) \phi_k(x)
	\dx
	\\
	&\quad =
	\int_{{\R^N}}  \sum_{k\in\I_s}
	 \mathcal{R}[Q_k^{1/p} v_k] (x) \:
	\left((\L - k^2) \left[ \int_{\T} \e{\i k t} \Phi(t,\cdot) \dt 	\right]\right)(x)
	\dx
	\\
	&\quad =
	\int_{\R^N}  \sum_{k\in\I_s}
	 \mathcal{R}[Q_k^{1/p} v_k] (x)	 \:\:
	\left[  \int_{\T}  \e{\i k t} \L \left[  \Phi(t,\cdot) \right](x)
	+ \partial_{tt} \left[ \e{\i k t} \right] \Phi(t,x) \dt 	\right]	\dx.
	\intertext{We now integrate by parts. Since $\Phi (\, \cdot \, , x)$ is periodic, the
	boundary terms in
	$$
	\int_{0}^{2\pi} \Phi \: \partial_{tt} \left[ \e{\i k t} \right] \dt
	= \left[ \i k \Phi \cdot \e{\i k t} + (\partial_t \Phi) \cdot \e{\i k t} \right]_{0}^{2\pi}
	+ \int_{0}^{2\pi}  \e{\i k t} \partial_{tt} \Phi \dt
	$$
	vanish for a.e. $x \in {\R^N}$. So we get}
	&\int_{\T \times {\R^N}} Q(t,x)|U(t,x)|^{p-2}U(t,x) \: \Phi(t,x)  \dtx \\
    &\quad =
	\int_{{\R^N}}  \sum_{k\in\I_s}
	 \mathcal{R}[Q_k^{1/p} v_k] (x) \:
	\left[  \int_{\T} \e{\i k t} \left(\partial_{tt} +  \L \right) \Phi  \dt 	\right]
	\dx	\\
		&\quad =
	\int_{{\R^N}}  \sum_{k\in\I_s}    \int_{\T} \e{\i k t}
	 \mathcal{R}[Q_k^{1/p} v_k] (x) \:
	\left(\partial_{tt} + \L \right) \Phi \dt \dx \\
	&\quad =
	\int_{\T \times {\R^N}}
	   \sum_{k\in\I_s}
	 \e{\i k t}
	 \mathcal{R}[Q_k^{1/p} v_k] (x)  \:
	\left(\partial_{tt} + \L \right) \Phi
	 \dtx	\\
	&\quad =
	\int_{\T \times {\R^N}}
	(Q^{-1/p}\mathscr{R})[V] \:
	(\partial_{tt}+ \L)  \Phi
	\dtx
	\\
	&\quad =
	\int_{\T \times {\R^N}} U \:
	(\partial_{tt} + \L )  \Phi
	\dtx.
\end{align*}

\medskip

Next we  extend this identity to more general test functions. We claim that the above identity even holds
for
\begin{align*}
	\Phi(t,x):= \sum_{k\in\J} \e{-\i k t}  \, \phi_k(x)
	 \qquad
	 \J\subset\Z \text{ finite},\; \phi_k=\overline{\phi_{-k}}\in C_c^\infty({\R^N}) \, (k\in\J).
\end{align*}
Indeed, given that $\J$ is finite, we have $\Phi\in L^{p'}(\R^N,L^{p'}(\T))$. Moreover, since
the functions $U(\cdot,x)$ and $Q|U(\cdot,x)|^{p-2}U(\cdot,x)$ have the symmetry  indexed by $s$
for almost all $x\in{\R^N}$, the time-symmetry requirement for the  test function
 is not a true restriction.  In fact, it imposes extra assumptions on the $\phi_k$ only for $s\in\{2,3\}$, see the
explanations near~\eqref{eq:defn_modes}, but those restrictions are not necessary since integration of $\sin$
against $\cos$-functions over the interval $[0,2\pi]$ gives 0. Moreover, $U(\cdot,x)$ and
$Q|U(\cdot,x)|^{p-2}U(\cdot,x)$ are $L^2(\T)$-orthogonal to the modes $\e{-\i k t}$ with $k\in\Z\sm \I_s$.
So the nonlinear wave-type equation actually holds in the distributional sense for test functions $\Phi$ as
above. 

\medskip

It remains to pass to the limit $\J\nearrow \Z$ because of
$$
  \Phi(t,x)= \sum_{k\in\Z} \e{-\i k t}  \, \phi_k(x),\quad
  \phi_k(x):= \frac{1}{2\pi}\int_\T \e{\i kt}\Phi(x,t)\dt
  \quad\text{where } \Phi\in C_c^\infty(\R^N,C^\infty(\T)).
$$
To see that this passage is possible, we choose a compact set $K\subseteq {\R^N}$ such that $\Phi(\cdot,t)$s
and hence all the $\phi_k$ have support contained in $K$. Then
\begin{align} \label{eq:DistributionalSolution}
  \begin{aligned}
	&\sum_{k\in\Z} \left| \int_{\T \times K} U \:
	(\partial_{tt} + \L )  \left(\e{-\i k t}  \, \phi_k(x)\right)\dtx\right| \\
	&\leq \sum_{k\in\Z}  \int_{\T \times K} |U| |(\L-k^2)\phi_k|\dtx \\
	&\leq \sum_{k\in\Z}  \|U\|_{L^q(K,L^q(\T))} \| (\L-k^2)\phi_k\|_{L^{q'}(K,L^{q'}(\T))} \\
	&\leq \sum_{k\in\Z}  \|U\|_{L^q(\R^N,L^q(\T))} (2\pi)^{1/q'} (\|\L\phi_k\|_{L^{q'}(K)}
	+ k^2 \|\phi_k\|_{L^{q'}(K)})\\
	&\stackrel{(A1)}\leq  \|U\|_{L^q(\R^N,L^q(\T))} (2\pi)^{1/q'} \sum_{k\in\Z}  (\|\phi_k\|_{W^{m,\infty}(K)}
	+ k^2|K|^{\frac{1}{q'}} \|\phi_k\|_{L^\infty(K)})
\end{aligned}
\end{align}
Since $\Phi$ is smooth and periodic with respect to $t$, we get from integration by parts
\begin{align*}
	\|\phi_k\|_{W^{m,\infty}(K)}
	&\leq  2(k^2+1)^{-1}  (\|\Phi\|_{W^{m,\infty}(\T\times
	K)} +  \|\partial_{tt}\Phi\|_{W^{m,\infty}(\T\times K)}), \\
	\|\phi_k\|_{L^{\infty}(K)}
	&\leq  2(k^4+1)^{-1}  (\|\Phi\|_{W^{m,\infty}(\T\times
	K)} + \|\partial_{tttt}\Phi\|_{W^{m,\infty}(\T\times K)})
\end{align*}
for all $k\in\Z$. As a consequence, this above series converges.
This shows that we can pass to the limit $\J\nearrow\Z$ on the right hand side of
the above distributional formulation of the wave equation. A similar estimate for the left hand side
gives that the nonlinear wave equation is satisfied in the sense of~\eqref{eq_weaksolution}, which finishes the proof.

\hfill $\square$

\section*{Acknowledgements}

The authors thank Wolfgang Reichel (KIT) for several discussions leading to an improvement of the manuscript.
Funded by the Deutsche Forschungsgemeinschaft (DFG, German Research Foundation) -- Project-ID 258734477 --
SFB 1173.

\bibliographystyle{abbrv}
\bibliography{Literatur}
  
\end{document}